\documentclass[12pt, reqno]{amsart}
\usepackage{amsmath, amsthm, amscd, amsfonts, amssymb, graphicx, color}
\usepackage{mathrsfs}
\usepackage[bookmarksnumbered, colorlinks, plainpages]{hyperref}
\hypersetup{colorlinks=true,linkcolor=red, anchorcolor=green, citecolor=cyan, urlcolor=red, filecolor=magenta, pdftoolbar=true}
\newtheorem{theorem}{Theorem}[section]
\newtheorem{lemma}[theorem]{Lemma}

\newtheorem{cor}[theorem]{Corollary}

\theoremstyle{remark}
\newtheorem{remark}[theorem]{\bf{Remark}}
\numberwithin{equation}{section}
\allowdisplaybreaks
\begin{document}
	
	\title [Berezin number and Berezin norm inequalities]{Berezin number and Berezin norm inequalities via Moore-Penrose inverse}
	\author[S. Mahapatra, A. Sen, R. Birbonshi, K. Paul]{Saikat Mahapatra, Anirban Sen, Riddhick Birbonshi, Kallol Paul}
	
	\address[Mahapatra]{Department of Mathematics, Jadavpur University, Kolkata 700032, West Bengal, India}
	\email{smpatra.lal2@gmail.com}
	
	\address[Sen] {Department of Mathematics, Jadavpur University, Kolkata 700032, West Bengal, India}
	\email{anirbansenfulia@gmail.com}

	\address[Birbonshi] {Department of Mathematics, Jadavpur University, Kolkata 700032, West Bengal, India}
	\email{riddhick.math@gmail.com}
	
	\address[Paul]{Vice-Chancellor, Kalyani University \& Professor of Mathematics\\ Jadavpur University (on lien) \\ Kolkata \\ West Bengal\\ India}
	\email{kalloldada@gmail.com}
	
	%\thanks{Mr. Anirban Sen would like to thank CSIR, Govt. of India, for the financial
		%	support in the form of Senior Research Fellowship under the mentorship of
		%	Prof. Kallol Paul. Mr. Subhadip Halder would like to thank UGC, Govt. of India for the financial support in the form of Junior Research Fellowship under the mentorship of Dr. Riddhick Birbonshi.}

	\subjclass[2020]{47A05,47B15,47B32.}
	
	\keywords{Berezin number, Berezin norm, Reproducing kernel Hilbert space, Moore-Penrose inverse, Bounded operator, Inequality.}

	\maketitle	
	
	\begin{abstract} 
		In this article, we establish the Berezin number and Berezin norm inequalities for bounded linear operators on a reproducing kernel Hilbert space using the Moore-Penrose inverse. The inequalities obtained here refine and generalize the earlier inequalities.
	\end{abstract}
	
	\section{Introduction}
	Let $ \mathscr{H}$ be a complex separable Hilbert space with inner product $\langle \cdot, \cdot \rangle$ and let $\|\cdot\|$ be the norm induced by the inner product. Let $ \mathscr{B}( \mathscr{H})$ be the $C^*$-algebra of all bounded linear operators on $ \mathscr{H}.$
	An operator $A \in  \mathscr{B}( \mathscr{H})$ is a positive operator if $\langle Ax,x \rangle \ge 0$ for all $x \in  \mathscr{H}$ and denoted by $A\ge 0.$ For $A \in  \mathscr{B}( \mathscr{H}),$ $|A|$ stands for the positive operator $(A^*A)^{1/2}.$
	Let $ \mathscr{CR}( \mathscr{H})$ be the set defined as 
	$ \mathscr{CR}( \mathscr{H})=\{A \in  \mathscr{B}( \mathscr{H}) : \text{range of $A$ is closed}\}.$
	If $A \in  \mathscr{CR}( \mathscr{H}),$ then there exists a unique $A^{\dag} \in  \mathscr{B}( \mathscr{H})$ which satisfies the following relations:
	\begin{align*}
		(a)~~AA^{\dag}A=A,
		(b)~~A^{\dag}AA^{\dag}=A^{\dag},
		(c)~~(AA^{\dag})^*=AA^{\dag},
		(d)~~(A^{\dag}A)^*=A^{\dag}A.
	\end{align*}
	Here, the operator $A^{\dag}$ is called the Moore-Penrose inverse of $A.$ For further information on the Moore-Penrose inverse, we refer to \cite{AB_02, Pen_55}.

	A reproducing kernel Hilbert space $ \mathscr{H}$ on a non-empty set $X$ is a Hilbert space of all complex-valued functions on $X$ with the property that for every $x \in X$, the corresponding linear evaluation functional on $ \mathscr{H}$ given by $\phi \to \phi(x),$ is continuous (see \cite{PR_Book_2016}). According to the Riesz representation theorem for each $x \in X$, there exists a unique $ k_x \in \mathscr{H}$ such that $\phi(x)=\langle \phi, k_x \rangle,$ for all $\phi \in  \mathscr{H}$. Here, $ k_x$ is the reproducing kernel at point $x$ and the collection $\{k_x :  x \in X \}$ is the set of all reproducing kernels of $ \mathscr{H}.$ The normalized reproducing kernel at the pont $x$ is given by $\hat{k}_{x} = k_x/\|k_x\|$ and the collection $\{\hat{k}_{x} :  x \in X \}$ is the set of all normalized reproducing kernels of $ \mathscr{H}.$
	
	If $\mathscr{H}$ is a reproducing kernel Hilbert space then for every $A \in  \mathscr{B}( \mathscr{H}),$ the Berezin transform of $A,$ (see \cite{BER1,BER2}) is the function $\widetilde{A}$ on $X$ defined by
	$$\widetilde{A}(x)=\langle A\hat{k}_{x},\hat{k}_{x} \rangle~~\text{for all $x \in X$}.$$
	The Berezin range of $A$ was introduced in \cite{K_CAOT_2013} as 
	$$\textbf{Ber}(A):= \{\widetilde{A}(x) : x \in X\}.$$
	For $A \in  \mathscr{B}( \mathscr{H}),$ the quantities $\textbf{ber}(A)$ and $\|A\|_{ber},$ referred to as the Berezin number and the Berezin norm of 
	$A$, respectively, are given by  
	$$\textbf{ber}(A):=\sup \left\{|\widetilde{A}(x)| : x \in X \right\}$$
	and 
	$$\|A\|_{ber}:=\sup\left\{|\langle A\hat{k}_{x},\hat{k}_{y}\rangle | : x,y \in X\right\},$$
	see \cite{BY_JIA_2020,KAR_JFA_2006}. Furthermore, if $A$ is positive, then $\|A\|_{ber}=\textbf{ber}(A),$ see \cite{BPS_CAOT_23}.
	A reproducing kernel Hilbert space possesses the ``Ber property", if the Berezin transform completely characterizes the operator (see \cite{KI_MFAT_2013}). It is well known that the Hardy-Hilbert space is a reproducing kernel Hilbert space which has the ``Ber property" (see \cite{ZHU_book_2007}). Recall that, the Hardy-Hilbert space on the unit disk $\mathbb{D}$ is denoted by $H^2(\mathbb{D}),$ and is defined as 
	$$H^2(\mathbb{D})=\left\{\phi : \phi(z)=\sum_{n=0}^{\infty}a_nz^n~~\text {with}~~ \sum_{n=0}^{\infty}|a_n|^2 ~~\text{is finite}\right\},$$
	see \cite{RR_book_2007}.
	Here we note that the Berezin number generally does not define a norm on $\mathscr{B}(\mathscr{H}),$ unless $\mathscr{H}$ possesses the ``Ber property". Over the years, many mathematicians have studied the Berezin number and Berezin norm inequalities of reproducing kernel Hilbert space operators, and we refer to \cite{NFAO,GA_CAOT_2021,HLB_LAMA_2020,MMM_JMAA_23,RMS,YI_LAMA_2022,YG_NYJM_2017}. 
	
	In this article, we obtain the Berezin number and Berezin norm inequalities of bounded linear operators acting on reproducing kernel Hilbert space by using the Moore-Penrose inverse. In comparison with concrete examples, we also observe that the obtained inequalities provide sharper bounds than the existing results.
	
	\section{Preliminaries}
	We begin this section with some auxiliary results that will be used repeatedly to obtain our main goal.  The first lemma is an easy extension of the classical Jensen and Young 
	inequalities.
	\begin{lemma}\cite{hardy1952inequalities}\label{cat}
		Let $\alpha,\beta$ be two positive real numbers.  Then for any $r\ge 1$,
		\begin{eqnarray*}
			(i)~~ \alpha^{\lambda} \beta^{\mu}&\le &\lambda \alpha+ \mu \beta\le \big(\lambda \alpha^r+\mu \beta^r\big)^\frac{1}{r}, \\
			(ii)~~ \alpha\beta&\le& \frac{\alpha^{\gamma}}{\gamma}+\frac{\beta^{\delta}}{\delta}\le \bigg(\frac{\alpha^{\gamma r}}{\gamma}+\frac{\beta^{\delta r}}{\delta}\bigg)^{\frac{1}{r}} ,
		\end{eqnarray*}
		where  $0\le \lambda,\mu\le 1$, $\lambda+\mu=1$ and $\gamma,\delta>1$, $\frac{1}{\gamma}+\frac{1}{\delta}=1$.
	\end{lemma}
	The second lemma is called Bohr's inequality, and it deals with a finite number of positive numbers.
	
	\begin{lemma}\cite{vasic1971some}
		Let $\alpha_i$ $(where\,\, i=1,2,\cdots,n )$ be positive real numbers. Then for any $r\ge 1$ \label{mit}\[\bigg(\sum_{i=1}^n \alpha_i\bigg)^r\le n^{r-1}\sum_{i=1}^{n}\alpha^r_i.\]
	\end{lemma}
	The third lemma is a consequence of Jensen's inequality and the spectral theorem for positive operators.
	
	\begin{lemma}\label{dog}
		\cite{furuta2005mond} If $A\in \mathscr{B}( \mathscr{H})$ is  positive  and $u\in  \mathscr{H}$ with $\| u\|=1$ then 
		\begin{eqnarray*}
			(i)~~ \langle Au,u\rangle^r&\le& \langle A^ru,u\rangle \,\,\,\mbox{ for $r\ge 1$},\\
			(ii)~~ \langle Au,u\rangle^r &\ge& \langle A^ru,u\rangle\,\,\,\mbox{ for $0<r\le 1$}.
		\end{eqnarray*}
	\end{lemma}
	The fourth lemma is  known as the Buzano's inequality.
	\begin{lemma}\cite{buzano1971generalizzazione}
		If $u,v,w\in\mathscr{H}$ and $\| w\| =1$, then the following inequality holds: \label{lmc} \[
		|\langle u,w\rangle \langle w,v\rangle|\le\frac{1}{2} \bigg(\langle u,u\rangle^{\frac{1}{2}}\langle v,v \rangle^{\frac{1}{2}}+|\langle u,v\rangle|\bigg).
		\] 
	\end{lemma} 
	The fifth lemma is called as the mixed Schwarz inequality.
	\begin{lemma}
		\cite{halmos2012hilbert}  If $A\in \mathscr{B}( \mathscr{H})$ and $u\in  \mathscr{H}$, then\label{nnnm}\[|\langle Au,u\rangle|^2\le\langle |A|u,u\rangle\langle |A^*|u,u\rangle.\]
	\end{lemma}
	The following lemma is an inner product inequality that is involved in the
	Moore-Penrose inverse of an operator.
	\begin{lemma}
		\cite{sababheh2024numerical} Let $A\in \mathscr{CR}( \mathscr{H})$ and $u,v\in \mathscr{H}$\label{bandor}. Then
		\[|\langle Au,v\rangle|^2\le\langle |A|^2u,u\rangle \langle AA^{\dag}v,v\rangle.\]
	\end{lemma}
	The next lemma is an inner product inequality that depends on a $2\times 2$ positive block matrix.
	
	\begin{lemma}\cite{kittaneh1988notes}\label{mm}
		Let $A_1,A_2,A_3\in  \mathscr{B}( \mathscr{H})$ such that $A_1$ and $A_2$ are both positive. Then $|\langle A_3u,v\rangle|^2\le \langle A_1u,u\rangle\langle A_2v,v\rangle$ holds $\forall$ $u,v\in\mathscr{H}$ if and only if $\begin{bmatrix}
			A_1 & A^*_3\\
			A_3 & A_2
		\end{bmatrix} $ is positive in $ \mathscr{B(H}\oplus  \mathscr{H)}$.
	\end{lemma}
	
	\section{Main Result}
	In our first result, we derive the following upper bound of the Berezin number.
	\begin{theorem}
		Let $A\in  \mathscr{C} \mathscr{R(H)}$ and $\lambda\in [0,1]$. Then for any $r\ge 1$\label{assa}
		\begin{eqnarray*}
			\textbf{ber}^{2r}(A)&\le& \min\bigg\{ \bigg\|  \frac{\lambda}{2}\big(|A|^{4r}+(AA^{\dag})^r\big)+(1-\lambda)|A^*|^{2r}\bigg\|_{ber},\\
			&&\bigg\| \frac{\lambda}{2}\big(|A^*|^{4r}+(A^{\dag}A)^r\big)+(1-\lambda)|A|^{2r} \bigg\|_{ber}\bigg\}.
		\end{eqnarray*}
		
	\end{theorem}

	\begin{proof} Suppose that $\hat{k}_x$ be a normalized reproducing kernel of $ \mathscr{H}$. Then
		\begin{eqnarray*}
			&&|\langle A\hat{k}_x,\hat{k}_x\rangle|^{2r}  \\
			&=& \lambda |\langle A\hat{k}_x,\hat{k}_x\rangle|^{2r}+(1-\lambda)|\langle A\hat{k}_x,\hat{k}_x\rangle|^{2r} \\
			&=& \lambda |\langle A\hat{k}_x,\hat{k}_x\rangle|^{2r}+(1-\lambda)|\langle A^*\hat{k}_x,\hat{k}_x\rangle|^{2r} \\
			&\le& \lambda \bigg(\langle |A|^2\hat{k}_x,\hat{k}_x\rangle\langle AA^{\dag}\hat{k}_x,\hat{k}_x\rangle\bigg)^{r}+(1-\lambda)|\langle A^*\hat{k}_x,A^*\hat{k}_x\rangle|^{r}\,\Big(\mbox{ by Lemma \ref{bandor}}\Big)\\
			&\le& \frac{\lambda}{2}\bigg(\langle |A|^2\hat{k}_x,\hat{k}_x\rangle^{2r}+\langle AA^{\dag}\hat{k}_x,\hat{k}_x\rangle^{2r}\bigg)+(1-\lambda)|\langle A^*\hat{k}_x,A^*\hat{k}_x\rangle|^{r}\\
			&\le& \frac{\lambda}{2}\bigg(\langle |A|^{4r}\hat{k}_x,\hat{k}_x\rangle+\langle (AA^{\dag})^r\hat{k}_x,\hat{k}_x\rangle\bigg)+(1-\lambda)\langle |A^*|^{2r}\hat{k}_x,\hat{k}_x\rangle\\&&\,\,\,\,\,\,\,\,\,\,\,\,\,\,\,\,\,\,\,\,\,\,\,\,\,\,\,\,\,\,\,\,\,\,\,\,\,\,\,\,\,\,\,\,\,\,\,\,\,\,\,\,\,\,\,\,\,\,\,\,\,\,\,\,\,\,\,\,\,\,\,\,\,\,\,\,\,\,\,\,\,\,\,\,\,\,\,\,\,\,\,\,\,\,\,\,\,\,\,\,\,\,\,\,\,\,\,\,\,\,\,\,\,\,\,\,\,\,\,\,\,\,\,\,\,\,\,\,\,\,\,\,\,\,\,\,\,\,\,\Big(\mbox{ by Lemma \ref{dog}}\Big)\\
			&=& \big\langle\big( \frac{\lambda}{2}\big(|A|^{4r}+(AA^{\dag})^r\big)+(1-\lambda)|A^*|^{2r}\big) \hat{k}_x,\hat{k}_x\big\rangle\\
			&\le & \big\| \frac{\lambda}{2}\big(|A|^{4r}+(AA^{\dag})^r\big)+(1-\lambda)|A^*|^{2r}\big \|_{ber}.
		\end{eqnarray*}
		Now, taking supremum over all $x\in X$, we get 
		\begin{eqnarray}
			\textbf{ber}^{2r}(A)&\le&\big\|  \frac{\lambda}{2}\big(|A|^{4r}+(AA^{\dag})^r\big)+(1-\lambda)|A^*|^{2r}\big\|_{ber}.\label{asd11}
		\end{eqnarray}
		Replacing $A$ by $A^*$, we get
		\begin{eqnarray}
			\textbf{ber}^{2r}(A)&\le& \| \frac{\lambda}{2}\big(|A^*|^{4r}+(A^{\dag}A)^r\big)+(1-\lambda)|A|^{2r} \|_{ber}.\label{asd1}
		\end{eqnarray} Combining (\ref{asd11}) and (\ref{asd1}) we get the
		the required result.
	\end{proof}
	\begin{remark}
		In \cite[Th. 2.6]{bhunia2023development}, the following bound is derived, i.e,
		\begin{equation}
			\textbf{ber}^{2r}(A)\le \textbf{ber}\big( \lambda |A|^{2r}+(1-\lambda)|A^*|^{2r}\big).\label{assa2}
		\end{equation}

		If we take $ \mathscr{H}=\mathbb{C}^2$,  $A=\begin{bmatrix}
			\frac{ 1}{2} & \frac{ 1}{2} \\
			0 & 0 
		\end{bmatrix}$, $r=1$ and $\lambda=\frac{1}{2}$ then inequality (\ref{assa2}) gives $\textbf{ber}^2(A) \le 0.375$, whereas Theorem \ref{assa} gives $\textbf{ber}^2(A) \le 0.3125$. Therefore, The upper bound provided by  Theorem \ref{assa} is better  than  the inequality  (\ref{assa2}).
	\end{remark}
	
	The next result reads as:
	
	\begin{theorem}
		Let $A\in \mathscr{C} \mathscr{R}( \mathscr{H})$. Then for any $r\ge 1$\label{vbc}
		\[\textbf{ber}^{2r}(A)\le\frac{1}{4}\| |A|^r+|A^*|^r\|_{ber} \|(U|A|U)^r+(V|A^*|V)^r\|_{ber},\] where $U=A^{\dag}A$ and $V=AA^{\dag}$.
	\end{theorem}
	\begin{proof}   Let $\hat{k}_x$ be a normalized reproducing kernel of $ \mathscr{H}$.
		Since $A=AU$ then
		from Lemma \ref{nnnm} we get \begin{eqnarray*}
			|\langle A\hat{k}_x,\hat{k}_x\rangle|^2= |\langle AU\hat{k}_x,\hat{k}_x\rangle|^2
			\le \langle U|A|U\hat{k}_x, \hat{k}_x\rangle \langle |A^*|\hat{k}_x,\hat{k}_x \rangle.
		\end{eqnarray*}
		Similarly,
		\[ |\langle A\hat{k}_x,\hat{k}_x\rangle|^2\le \langle V|A^*|V\hat{k}_x, \hat{k}_x\rangle \langle |A|\hat{k}_x,\hat{k}_x \rangle. \]
		Therefore, \begin{eqnarray*}
			&& |\langle A\hat{k}_x,\hat{k}_x \rangle|^{2r}\\
			&\le & \big(\langle U|A|U\hat{k}_x, \hat{k}_x\rangle \langle |A^*|\hat{k}_x,\hat{k}_x \rangle\big)^{\frac{r}{2}}\big(  \langle V|A^*|V\hat{k}_x, \hat{k}_x\rangle \langle |A|\hat{k}_x,\hat{k}_x \rangle  \big)^{\frac{r}{2}}\\
			&\le & \frac{1}{4}\big(\langle |A|\hat{k}_x,\hat{k}_x \rangle^r + \langle |A^*|\hat{k}_x,\hat{k}_x \rangle^r\big)\big(  \langle V|A^*|V\hat{k}_x, \hat{k}_x\rangle^r + \langle U|A|U\hat{k}_x, \hat{k}_x\rangle^r\big )\\
			&\le & \frac{1}{4}\big(\langle |A|^r\hat{k}_x,\hat{k}_x \rangle + \langle |A^*|^r\hat{k}_x,\hat{k}_x \rangle\big)\big(  \langle( V|A^*|V)^r\hat{k}_x, \hat{k}_x\rangle + \langle (U|A|U)^r\hat{k}_x, \hat{k}_x\rangle\big )\\
			&&\,\,\,\,\,\,\,\,\,\,\,\,\,\,\,\,\,\,\,\,\,\,\,\,\,\,\,\,\,\,\,\,\,\,\,\,\,\,\,\,\,\,\,\,\,\,\,\,\,\,\,\,\,\,\,\,\,\,\,\,\,\,\,\,\,\,\,\,\,\,\,\,\,\,\,\,\,\,\,\,\,\,\,\,\,\,\,\,\,\,\,\,\,\,\,\,\,\,\,\,\,\,\,\,\,\,\,\,\,\,\,\,\,\,\,\,\,\,\,\,\,\,\,\,\,\,\,\,\,\,\,\,\,\,\,\,\,\,\,\,\,\,\Big(\mbox{ by Lemma \ref{dog}}\Big)\\
			&\le& \frac{1}{4}\big\| |A|^r+|A^*|^r\|_{ber} \|(U|A|U)^r+(V|A^*|V)^r\big\|_{ber}.
		\end{eqnarray*}
		Therefore, we obtain the required inequality by taking supremum over all $x\in X$.
	\end{proof}
	
	\begin{remark}
		In particular, if $A$ is invertible, then \cite[Th. 2.2]{guesba2025further} becomes
		\begin{eqnarray}
			\textbf{ber}(A)&\le& \frac{1}{2}\textbf{ber}(|A|^2+I)\nonumber\\
			&=& \frac{1}{2}\||A|^2+I\|_{ber}\label{vbc1}.
		\end{eqnarray}
		For $r=1$ Theorem \ref{vbc} becomes \begin{eqnarray}
			\textbf{ber}(A)
			&\le& \frac{1}{2}\||A|+|A^*|\|_{ber}\label{vbc2}.
		\end{eqnarray}
		Now,  if we take $ \mathscr{H}=\mathbb{C}^2$ and  $A=\begin{bmatrix}
			2 & 0\\
			0 & 1
		\end{bmatrix}$ then inequality (\ref{vbc1}) gives $\textbf{ber}(A)\le 2.5$,  while the inequality (\ref{vbc2}) gives $\textbf{ber}(A)\le 2$. Therefore, inequality (\ref{vbc2}) gives a better bound than inequality (\ref{vbc1}). 
	\end{remark}
	
	The following result gives another estimate for the Berezin number.
	
	\begin{theorem}
		Let $A\in  \mathscr{CR}( \mathscr{H})$ and $\lambda\in[0,1] $. Then for any $r\ge 1$ \label{th10}\[
		\textbf{ber}^{2r}(A)\le \frac{\lambda}{2}\||A|^{4r}+(AA^{\dag})^r\|_{ber}+\frac{(1-\lambda)}{2}\textbf{ber}^r(A)\||A|^{2r}+(AA^{\dag})^r\|_{ber}.
		\]
	\end{theorem}
	\begin{proof} Let $\hat{k}_x$ be a normalized reproducing kernel of $ \mathscr{H}$. Then  we get
		\begin{eqnarray*}
			&& |\langle A\hat{k}_x,\hat{k}_x \rangle|^{2r}\\
			&=&\lambda|\langle A\hat{k}_x,\hat{k}_x \rangle|^{2r} +(1-\lambda) |\langle A\hat{k}_x,\hat{k}_x \rangle|^{2r}\\
			&\le & \lambda \langle |A|^2\hat{k}_x,\hat{k}_x\rangle^r \langle AA^{\dag}\hat{k}_x,\hat{k}_x\rangle^r\\
			&&\,\,\,\,\,\,+(1-\lambda)|\langle A\hat{k}_x,\hat{k}_x\rangle|^r\langle |A|^2\hat{k}_x,\hat{k}_x\rangle^{\frac{r}{2}}\langle AA^{\dag}\hat{k}_x,\hat{k}_x\rangle^{\frac{r}{2}}\,\,\,\,\,\,\,\,\,\Big(\mbox{by Lemma \ref{bandor}}\Big)\\
			&\le & \frac{\lambda}{2}\bigg( \langle |A|^2\hat{k}_x,\hat{k}_x\rangle^{2r}+ \langle AA^{\dag}\hat{k}_x,\hat{k}_x\rangle^{2r}\bigg)\\
			&&\,\,\,\,\,\,\,\,+\frac{(1-\lambda)}{2}|\langle A\hat{k}_x,\hat{k}_x\rangle|^r\bigg(\langle |A|^2\hat{k}_x,\hat{k}_x\rangle^r+\langle AA^{\dag}\hat{k}_x,\hat{k}_x\rangle^r\bigg)\\
			&\le & \frac{\lambda}{2}\big \langle\big (|A|^{4r}+(AA^{\dag})^r\big)\hat{k}_x,\hat{k}_x\big\rangle\\
			&&\,\,\,\,\,\,\,\,+\frac{(1-\lambda)}{2}|\langle A\hat{k}_x,\hat{k}_x\rangle|^r\big\langle \big(|A|^{2r}+(AA^{\dag})^r\big)\hat{k}_x,\hat{k}_x\big\rangle\,\,\,\,\,\,\,\,\,\Big(\mbox{by Lemma \ref{dog}}\Big)\\
			&\le& \frac{\lambda}{2}\big\||A|^{4r}+(AA^{\dag})^r\big\|_{ber}+\frac{(1-\lambda)}{2}\textbf{ber}^r(A)\big\||A|^{2r}+(AA^{\dag})^r\big\|_{ber}.
		\end{eqnarray*}
		Therefore, the desired inequality is obtained by taking supremum over all $x\in X$.
	\end{proof}
	\begin{remark}
		In \cite[Cor. 2.10]{sen2025berezin}, it was given that \begin{eqnarray}
			\textbf{ber}^4(A)\le \frac{1}{4}\||A|^4+|A^*|^4\|_{ber}+\frac{1}{4}\|A^2\|^2_{ber}.\label{gfd}\end{eqnarray} If we take $ \mathscr{H}=\mathbb{C}^2$ and $A=\begin{bmatrix}
			1 & 1 & 0\\
			0 & 0 & 0\\
			1 & 1 & 0
		\end{bmatrix}$ then the bound in (\ref{gfd}) gives $\textbf{ber}^4(A)\le 4.25 $, whereas Theorem  \ref{th10} gives $\textbf{ber}^4(A)\le \frac{1}{4}\big(6\lambda +2.5\big)^2 $ for $r=1$. Therefore, for any $\lambda\le \frac{1}{4}$ the inequality in Theorem  \ref{th10}  gives better upper bound than the inequality in (\ref{gfd}).
	\end{remark}
	
	We now observe the following scalar inequality
	\begin{equation}
		|\alpha+\beta|\le \sqrt{2}|\alpha+i\beta|,\,\,\,where\,\, \alpha,\beta\in\mathbb{R}.\label{goru}
	\end{equation}
	The following upper bound of the Berezin number for the sum of two operators is obtained using this inequality.
	
	\begin{theorem}
		Let $A_1,A_2\in  \mathscr{CR}( \mathscr{H})$. Then \label{nnbb}
		\[\textbf{ber}(A_1+A_2)\le \frac{1}{\sqrt{2}}\textbf{ber}\bigg(|A_1|^2+|A_2|^2+i\big(A_1A^{\dag}_1+A_2A^{\dag}_2\big)\bigg).\]
	\end{theorem}
	\begin{proof}
		Let $\hat{k}_x$ be a normalized reproducing kernel of $ \mathscr{H}$. Then we get
		\begin{eqnarray*}
			&&|\langle (A_1+A_2)\hat{k}_x,\hat{k}_x\rangle|\\&\le& |\langle A_1\hat{k}_x,\hat{k}_x \rangle|+|\langle A_2 \hat{k}_x,\hat{k}_x \rangle|\\
			&\le & \langle |A_1|^2\hat{k}_x,\hat{k}_x \rangle^{\frac{1}{2}}\langle A_1A^{\dag}_1\hat{k}_x,\hat{k}_x \rangle^{\frac{1}{2}}+\langle |A_2|^2\hat{k}_x,\hat{k}_x \rangle^{\frac{1}{2}}\langle A_2A^{\dag}_2\hat{k}_x,\hat{k}_x\rangle^{\frac{1}{2}}\\
			&&\,\,\,\,\,\,\,\,\,\,\,\,\,\,\,\,\,\,\,\,\,\,\,\,\,\,\,\,\,\,\,\,\,\,\,\,\,\,\,\,\,\,\,\,\,\,\,\,\,\,\,\,\,\,\,\,\,\,\,\,\,\,\,\,\,\,\,\,\,\,\,\,\,\,\,\,\,\,\,\,\,\,\,\,\,\,\,\,\,\,\,\,\,\,\,\,\,\,\,\,\,\,\,\,\,\,\,\,\,\,\,\,\,\,\,\,\,\,\,\,\,\,\,\,\,\,\,\,\,\,\,\,\big(\mbox{by Lemma \ref{bandor}}\big)\\
			&\le & \frac{1}{2}\bigg(\langle |A_1|^2\hat{k}_x,\hat{k}_x \rangle+\langle A_1A^{\dag}_1\hat{k}_x,\hat{k}_x \rangle\bigg)+\frac{1}{2} \bigg(\langle |A_2|^2\hat{k}_x,\hat{k}_x \rangle+\langle A_2A^{\dag}_2\hat{k}_x,\hat{k}_x \rangle\bigg)\\
			&=& \big\langle\big(\frac{|A_1|^2+|A_2|^2}{2}\big)\hat{k}_x,\hat{k}_x\big\rangle+\big\langle\big(\frac{A_1A^{\dag}_1+A_2A^{\dag}_2}{2}\big)\hat{k}_x,\hat{k}_x\big\rangle\\
			&\le & \sqrt{2}\bigg|\big\langle\big(\frac{|A_1|^2+|A_2|^2}{2}\big)\hat{k}_x,\hat{k}_x\big\rangle+i\big\langle\big(\frac{A_1A^{\dag}_1+A_2A^{\dag}_2}{2}\big)\hat{k}_x,\hat{k}_x\big\rangle\bigg|\,\,\,\,\,\,\,\,\Big(\mbox{ by (\ref{goru} )}\Big)\\
			&=&\sqrt{2}\bigg|\big\langle\big(\frac{|A_1|^2+|A_2|^2}{2}+i \frac{A_1A^{\dag}_1+A_2A^{\dag}_2}{2}\big)\hat{k}_x,\hat{k}_x\big\rangle\bigg|\\
			&\le& \frac{1}{\sqrt{2}}\textbf{ber}\bigg(|A_1|^2+|A_2|^2+i\big(A_1A^{\dag}_1+A_2A^{\dag}_2\big)\bigg).
		\end{eqnarray*}
		So, taking supremum over all $x\in X$, we get the required inequality.
	\end{proof}
	\begin{remark}
		Following \cite[Th. 3.3]{huban2022some}, for the case $\epsilon=1$ and $\zeta=\frac{1}{2}$, we get\begin{equation}
			\textbf{ber}^2(A)\le \frac{1}{2}\||A|^2+|A^*|^2\|_{ber}.\label{xxcc}
		\end{equation}
		
		If $A_2=0$ and $A_1=A$ then it follows from Theorem \ref{nnbb}
		that     \[
		\textbf{ber}(A)\le \frac{1}{\sqrt{2}}\textbf{ber}\big(|A|^2+iAA^{\dag}\big).
		\]
		If we take $ \mathscr{H}=\mathbb{C}^2$ and  $A=\begin{bmatrix}
			1 & 1\\
			0 & 0
		\end{bmatrix}$, then the inequality (\ref{xxcc}) gives $\textbf{ber}^2(A)\le 1.5,$ while Theorem \ref{nnbb} gives $\textbf{ber}^2(A)\le 1$ Therefore, we remark that Theorem \ref{nnbb} provides a stronger bound than in the inequality (\ref{xxcc}).
	\end{remark}
	
	Now, we note the following scalar inequality
	\begin{eqnarray}
		(\alpha_1\beta_1+\alpha_2\beta_2)^2\le(\alpha^2_1+\alpha^2_2)(\beta^2_1+\beta^2_2)\,\,\,\,\forall\,\,\alpha_1,\alpha_2,\beta_1,\beta_2\in \mathbb{R}\label{sree}
	\end{eqnarray}
	The following  upper bound of the Berezin number for the sum of two operators is obtained using this inequality.

	\begin{theorem}\label{T1}
		Let $A_1,A_2\in \mathscr{CR(H)}$. Then
		\begin{eqnarray*}
			\textbf{ber}^2(A_1+A_2) 
			&\le&\textbf{ber}(|A_1|^2+iA_2A^{\dag}_2)\textbf{ber}(|A_2|^2+iA_1A^{\dag}_1)\\&&\,\,\,\,\,\,\,\,\,\,\,\,\,\,\,\,\,+\frac{1}{2}\||A_1|^2+ A_2A^{\dag}_2\|_{ber}\||A_2|^2+A_1A^{\dag}_1\|_{ber}.\end{eqnarray*}
	\end{theorem}
	\begin{proof}   Suppose that $\hat{k}_x$ be a normalized reproducing kernel of $ \mathscr{H}$. Then
		\begin{eqnarray*}
			&&|\langle (A_1+A_2)\hat{k}_x,\hat{k}_x\rangle|^2\\
			&\le& |\langle A_1\hat{k}_x,\hat{k}_x \rangle|^2+|\langle A_2\hat{k}_x,\hat{k}_x \rangle|^2+2|\langle A_1\hat{k}_x,\hat{k}_x\rangle||\langle A_2 \hat{k}_x,\hat{k}_x\rangle|\\
			&\le&  |\langle A_1\hat{k}_x,\hat{k}_x \rangle|^2+|\langle A_2\hat{k}_x,\hat{k}_x \rangle|^2+\frac{1}{2}\bigg(|\langle A_1\hat{k}_x,\hat{k}_x\rangle|+|\langle A_2\hat{k}_x,\hat{k}_x\rangle|\bigg)^2\\
			&\le& \langle |A_1|^2\hat{k}_x,\hat{k}_x\rangle \langle A_1A^{\dag}_1\hat{k}_x,\hat{k}_x\rangle+\langle |A_2|^2\hat{k}_x,\hat{k}_x\rangle \langle A_2A^{\dag}_2\hat{k}_x,\hat{k}_x\rangle\\
			&&+\frac{1}{2}\bigg(\langle |A_1|^2\hat{k}_x,\hat{k}_x\rangle ^{\frac{1}{2}}\langle A_1A^{\dag}_1\hat{k}_x,\hat{k}_x\rangle^{\frac{1}{2}}+\langle |A_2|^2\hat{k}_x,\hat{k}_x\rangle ^{\frac{1}{2}}\langle A_2A^{\dag}_2\hat{k}_x,\hat{k}_x\rangle^{\frac{1}{2}}\bigg)^2\\
			&& \,\,\,\,\,\,\,\,\, \,\,\,\,\,\,\,\,\, \,\,\,\,\,\,\,\,\, \,\,\,\,\,\,\,\,\,\,\, \,\,\,\,\,\,\,\,\, \,\,\,\,\,\,\,\,\, \,\,\,\,\,\,\,\,\,\,\,\,\,\,\, \,\,\,\,\,\,\,\,\,\,\,\,   \,\,\,\,\,\,\,\,\,\,\,\,\, \,\,\,\,\,\,\,\,\,  \,\,\,\,\,\,\,\,\,\,\,\,\,\,\,\,\,\,\,\,\, \,\,\,\,\,\,\,\,\, \,\,\,\,\,\,\,\,\,\Big(\mbox{ by Lemma \ref{bandor}}\Big)\\
			&\le&\bigg( \langle |A_1|^2\hat{k}_x,\hat{k}_x\rangle ^2+\langle A_2A^{\dag}_2\hat{k}_x,\hat{k}_x\rangle^2\bigg)^{\frac{1}{2}}\bigg(\langle |A_2|^2\hat{k}_x,\hat{k}_x\rangle^2 \langle A_1A^{\dag}_1\hat{k}_x,\hat{k}_x\rangle^2\bigg)^{\frac{1}{2}}\\
			&&+\frac{1}{2}\big(\langle |A_1|^2\hat{k}_x,\hat{k}_x\rangle +\langle A_2A^{\dag}_2\hat{k}_x,\hat{k}_x\rangle\big)\big(\langle |A_2|^2\hat{k}_x,\hat{k}_x\rangle +\langle A_1A^{\dag}_1\hat{k}_x,\hat{k}_x\rangle\big)\,\, \Big(\mbox{by (\ref{sree})}\Big)\\
			&=& \big|\langle (|A_1|^2+iA_2A^{\dag}_2)\hat{k}_x,\hat{k}_x\rangle \big|\big|\langle (|A_2|^2+iA_1A^{\dag}_1)\hat{k}_x,\hat{k}_x\rangle\big|\\
			&&+\frac{1}{2}\langle (|A_1|^2+ A_2A^{\dag}_2)\hat{k}_x,\hat{k}_x\rangle \langle (|A_2|^2+A_1A^{\dag}_1) \hat{k}_x,\hat{k}_x\rangle \\
			&\le & \textbf{ber}(|A_1|^2+iA_2A^{\dag}_2)\textbf{ber}(|A_2|^2+iA_1A^{\dag}_1)+\frac{1}{2}\||A_1|^2+ A_2A^{\dag}_2\|_{ber}\||A_2|^2+A_1A^{\dag}_1\|_{ber}.
		\end{eqnarray*}
		Therefore, we obtain the required bound by taking supremum over all $x\in X$.
	\end{proof}
	
	In particular, considering $A_1=A_2=A$ in Theorem \ref{T1} we get the following corollary.
	\begin{cor}
		Let $A\in  \mathscr{CR(H)}$. Then\label{ere}
		\[\textbf{ber}^2(A) \le\frac{1}{4}\textbf{ber}^2(|A|^2+iAA^{\dag})+\frac{1}{8}\||A|^2+ AA^{\dag}\|^2_{ber}.\]
	\end{cor}
	\begin{remark}
		In \cite[Cor. 2.18]{sen2025berezin}, it was given that \begin{eqnarray}
			\textbf{ber}^2(A)\le \frac{1}{4}\textbf{ber}^2\big(|A|+i|A^*|\big)+\frac{1}{8}\||A|+|A^*|\|^2_{ber}\label{gfd2}\end{eqnarray} If we take $ \mathscr{H}=\mathbb{C}^2$ and $A=\begin{bmatrix}
			1 & 1 \\
			0 & 0 
		\end{bmatrix}$ then the bound in (\ref{gfd2}) gives $\textbf{ber}^2(A)\le 1.1875$, while Corollary \ref{ere} gives $\textbf{ber}^2(A)\le 1  $. Therefore, the inequality in Corollary \ref{ere} is sharper  than the inequality in (\ref{gfd2}).
	\end{remark}

	To prove our next result we need the following lemma.

	\begin{lemma}\label{mm1}
		Let $A_1,A_2,A_3\in  \mathscr{B}( \mathscr{H})$ such that $A_1$ and $A_2$ are both positive. If $\begin{bmatrix}
			A_1 & A^*_3\\
			A_3 & A_2
		\end{bmatrix} \in  \mathscr{B}( \mathscr{H}\oplus  \mathscr{H})$ is positive  then for any $r\ge 1$ \[\textbf{ber}^{2r}(A_3)\le\frac{1}{2}\bigg(\|A^{2r}_1\|_{ber}\|A^{2r}_2\|_{ber}\bigg)^{\frac{1}{2}}+\frac{1}{2}\textbf{ber}^r(A_1A_2).\]
	\end{lemma}
	\begin{proof} Let $\hat{k}_x$ be a normalized reproducing kernel of $ \mathscr{H}$. From Lemma \ref{mm}, we have
		\begin{eqnarray*}
			&&|\langle A_3\hat{k}_x,\hat{k}_x \rangle|^2\\&\le& \langle A_1\hat{k}_x,\hat{k}_x \rangle \langle A_2\hat{k}_x,\hat{k}_x\rangle\\
			&=& \langle A_1\hat{k}_x,\hat{k}_x \rangle \langle \hat{k}_x,A_2\hat{k}_x\rangle\\
			&\le & \frac{1}{2} \bigg(\langle A_1\hat{k}_x,A_1\hat{k}_x\rangle \langle A_2\hat{k}_x,A_2\hat{k}_x\rangle \bigg)^{\frac{1}{2}}+\frac{1}{2}|\langle A_1\hat{k}_x,A_2\hat{k}_x\rangle|\,\,\,\,\,\,\,\,\,\,\,\,\,\,\,\,\,\Big(\mbox{by Lemma \ref{lmc}}\Big)\\
			&\le & \bigg(\frac{1}{2} \big(\langle A_1\hat{k}_x,A_1\hat{k}_x\rangle \langle A_2\hat{k}_x,A_2\hat{k}_x\rangle \big)^{\frac{r}{2}}+\frac{1}{2}|\langle A_1\hat{k}_x,A_2\hat{k}_x\rangle|^r\bigg)^{\frac{1}{r}}\,\,\,\,\,\,\Big(\mbox{by Lemma \ref{cat}}\Big)\\
			&\le & \bigg(\frac{1}{2} \big(\langle A^{2r}_1\hat{k}_x,\hat{k}_x\rangle \langle A^{2r}_2\hat{k}_x,\hat{k}_x\rangle \big)^{\frac{1}{2}}+\frac{1}{2}|\langle A_1\hat{k}_x,A_2\hat{k}_x\rangle|^r\bigg)^{\frac{1}{r}}\,\,\,\,\,\,\,\,\,\,\,\,\,\,\,\Big(\mbox{by Lemma \ref{dog}}\Big)\\
			&\le &\bigg( \frac{1}{2}\big(\|A^{2r}_1\|_{ber}\|A^{2r}_2\|_{ber}\big)^{\frac{1}{2}}+\frac{1}{2}\textbf{ber}^r(A_1A_2)\bigg)^{\frac{1}{r}}.
		\end{eqnarray*}
		Therefore, the required inequality is obtained by taking supremum over all $x\in X$.
	\end{proof}
	\begin{theorem}
		Let $A_1,A_2\in  \mathscr{CR}( \mathscr{H})$. Then for any $r\ge 1$ \label{hjh}
		\begin{eqnarray*}
			\textbf{  ber}^{2r}(A_1+A_2)&\le&\frac{1}{2}\bigg(\big\|\big( A^{ \dag}_1A_1+A^*_2A_2\big)^{2r}\big\|_{ber}\big\|\big(A_1A^{ *}_1+A_2A^{\dag}_2\big)^{2r}\big\|_{ber}\bigg)^{\frac{1}{2}}\nonumber\\&&+\frac{1}{2}\textbf{ber}^r\bigg(( A^{ \dag}_1A_1+A^*_2A_2)(A_1A^{ *}_1+A_2A^{\dag}_2)\bigg).
		\end{eqnarray*}
		
	\end{theorem}
	\begin{proof}
		The $2\times 2$ block matrices $\begin{bmatrix}
			A^{\dag}_1A_1 & A^*_1\\
			A_1 & A_1A^*_1
		\end{bmatrix}$ and $\begin{bmatrix}
			A^*_2A_2 & A^*_2\\
			A_2 & A_2A^{\dag}_2
		\end{bmatrix}$ are positive. Therefore, the sum of these two block matrices is
		$\begin{bmatrix}
			A^{ \dag}_1A_1+A^*_2A_2& (A_1+A_2)^*\\
			A_1+A_2 & A_1A^{ *}_1+A_2A^{\dag}_2 
		\end{bmatrix}$, which is also positive. From Lemma \ref{mm1}
		we obtain the desired inequality.      
	\end{proof}
	\begin{cor}
		Let $A\in  \mathscr{CR}( \mathscr{H})$. Then
		\begin{eqnarray*}
			\textbf{  ber}^{2r}(A)
			&\le& \frac{1}{2}\min\Bigg\{\bigg(\big\|\big( A^*A\big)^{2r}\big\|_{ber}\big\|\big(AA^{\dag}\big)^r\big\|_{ber}\bigg)^{\frac{1}{2}}+\textbf{ber}^r\big(A^*A^2A^{\dag}\big),\\
			&& \, \,\,\, \,\,  \,\, \,\, \,\,  \,\, \,\, \,\,\,\,\,\,\bigg(\big\|\big(AA^{ *}\big)^{2r}\big\|_{ber}\big\| \big(A^{ \dag}A\big)^r\big\|_{ber}\bigg)^{\frac{1}{2}}+\textbf{ber}^r\big( A^{ \dag}A^2A^{ *}\big)\Bigg\}.
		\end{eqnarray*}
	\end{cor}
	\begin{proof}  Considering $A_1=0$ and $A_2=0$ in Theorem \ref{hjh} respectively, we get
		\begin{eqnarray}
			\textbf{ber}^{2r}(A_2)&\le&\frac{1}{2}\bigg(\big\|\big( A^*_2A_2\big)^{2r}\big\|_{ber}\big\|(A_2A^{\dag}_2)^r\big\|_{ber}\bigg)^{\frac{1}{2}}+\frac{1}{2}\textbf{ber}^r\big(A^*_2A^2_2A^{\dag}_2\big)\label{hari0}\end{eqnarray}
		and
		\begin{eqnarray}
			\textbf{ ber}^{2r}(A_1)&\le&\frac{1}{2}\bigg(\big\|\big(A_1A^{ *}_1\big)^{2r}\big\|_{ber}\big\|( A^{ \dag}_1A_1)^r\big\|_{ber}\bigg)^{\frac{1}{2}}+\frac{1}{2}\textbf{ber}^r\big( A^{ \dag}_1A^2_1A^{ *}_1\big).\label{hari1}
		\end{eqnarray}
		Combining (\ref{hari0}) and (\ref{hari1}) we get our required bound.
	\end{proof}

	Now, we prove the following lemma to obtain our next result.

	\begin{lemma}
		Let $A_1,A_2,A_3\in  \mathscr{B}( \mathscr{H})$, where $A_1$ and $A_2$ are both positive.  If $\begin{bmatrix}
			A_1 & A^*_3\\
			A_3 & A_2
		\end{bmatrix} \in  \mathscr{B}( \mathscr{H}\oplus  \mathscr{H})$ is positive then for any $r\ge 1$
		\[\textbf{ber}^{2r}(A_3) \le \frac{1}{4}\|A^{2r}_1+A^{2r}_2\|_{ber}+\frac{1}{2}\textbf{ber}^{r}(A_1A_2).\]\label{virat}
	\end{lemma}
	\begin{proof}
		Suppose that $\hat{k}_x$ be a normalized reproducing kernel of $ \mathscr{H}$. Using Lemma \ref{mm} we get
		\begin{eqnarray*}
			&& |\langle A_3\hat{k}_x,\hat{k}_x \rangle|^{2}\\
			&\le& \langle A_1\hat{k}_x,\hat{k}_x \rangle \langle A_2\hat{k}_x,\hat{k}_x\rangle\\
			&=& \langle A_1\hat{k}_x,\hat{k}_x \rangle \langle \hat{k}_x,A_2\hat{k}_x\rangle\\
			&\le & \frac{1}{2} \sqrt{\langle A_1\hat{k}_x,A_1\hat{k}_x\rangle }\sqrt{\langle A_2\hat{k}_x,A_2\hat{k}_x\rangle }+\frac{1}{2}|\langle A_1\hat{k}_x,A_2\hat{k}_x\rangle|\,\,\,\,\,\,\,\,\,\,\,\,\,\,\Big(\mbox{ by Lemma \ref{lmc}}\Big)\\
			&\le & \bigg(\frac{1}{2}\langle  A_1\hat{k}_x,A_1\hat{k}_x\rangle^{\frac{r}{2}}\langle A_2\hat{k}_x,A_2\hat{k}_x\rangle^{\frac{r}{2}}+\frac{1}{2}|\langle A_1\hat{k}_x,A_2\hat{k}_x\rangle|^r\bigg)^{\frac{1}{r}}\,\,\,\,\,\Big(\mbox{ by Lemma \ref{cat}}\Big)\\
			&\le & \bigg(\frac{1}{4}\big(\langle A^{2r}_1\hat{k}_x,\hat{k}_x\rangle+\langle A^{2r}_2 \hat{k}_x,\hat{k}_x\rangle\big)+\frac{1}{2}\textbf{ber}^r(A_1A_2)\bigg)^{\frac{1}{r}}\,\,\,\,\,\,\,\,\,\,\,\,\,\,\,\,\Big(\mbox{ by Lemma \ref{dog}}\Big)\\
			&=&\bigg(\frac{1}{4}\langle \big(A^{2r}_1+A^{2r}_2 \big)\hat{k}_x,\hat{k}_x\rangle+\frac{1}{2}\textbf{ber}^r(A_1A_2)\bigg)^{\frac{1}{r}}\\
			& \le& \bigg(\frac{1}{4}\|A^{2r}_1+A^{2r}_2\|_{ber}+\frac{1}{2}\textbf{ber}^r(A_1A_2)\bigg)^{\frac{1}{r}}.
		\end{eqnarray*}So, taking supremum over all $x\in X$, we get the desired bound.
	\end{proof}
	\begin{theorem}
		Let $A_1,A_2\in \mathscr{CR}( \mathscr{H})$. Then for any $r\ge 1$
		\begin{eqnarray*}
			\textbf{ber}^{2r}(A_1+A_2)
			&\le& \frac{1}{4} \big\|( A^{ \dag}_1A_1+A^*_2A_2)^{2r}+(A_1A^{ *}_1+A_2A^{\dag}_2)^{2r}\big\|_{ber}\\
			&&+\frac{1}{2}\textbf{ber}^r\bigg(( A^{ \dag}_1A_1+A^*_2A_2)(A_1A^{ *}_1+A_2A^{\dag}_2)\bigg).
		\end{eqnarray*}\label{Rohit}
	\end{theorem}
	\begin{proof}
		Since, $A^{ \dag}_1A_1+A^*_2A_2$, $A_1A^{ *}_1+A_2A^{\dag}_2$  are positive so, the $2\times 2$ block matrix $\begin{bmatrix}
			A^{ \dag}_1A_1+A^*_2A_2 & (A_1+A_2)^*\\
			A_1+A_2 & A_1A^{ *}_1+A_2A^{\dag}_2
		\end{bmatrix}$ is also positive. Therefore, the desired bound follows from Lemma \ref{virat}.
	\end{proof}
	\begin{cor}
		Let $A\in  \mathscr{CR}( \mathscr{H})$. Then\label{ani1}
		\begin{eqnarray*}
			\textbf{ber}^{2r}(A)
			&\le& \min\bigg\{ \frac{1}{4}\big\||A^*|^{4r}+(A^{\dag}A)^r\big\|_{ber}+\frac{1}{2}\textbf{ber}^r(A^{\dag}A^2A^*),\\
			&&\hspace{0.8 cm}\frac{1}{4}\big\||A|^{4r}+(AA^{\dag})^r\big\|_{ber}+\frac{1}{2}\textbf{ber}^r(A^*A^2A^{\dag})\bigg\}.
		\end{eqnarray*}
		
	\end{cor}
	\begin{proof}
		Considering $A_1=0$ and $A_2=0$ in Theorem \ref{Rohit}, respectively, we obtain 
		\begin{eqnarray}
			\textbf{ ber}^{2r}(A_2)
			&\le& \frac{1}{4}\big\||A_2|^{4r}+(A_2A^{\dag}_2)^r\big\|_{ber}+\frac{1}{2}\textbf{ber}^r(A^*_2A^2_2A^{\dag}_2)\label{indm1}
		\end{eqnarray}
		and
		\begin{eqnarray}
			\textbf{  ber}^{2r}(A_1)
			&\le& \frac{1}{4}\big\||A^*_1|^{4r}+(A^{\dag}_1A_1)^r\|_{ber}+\frac{1}{2}\textbf{ber}^r(A^{\dag}_1A^2_1A^*_1)\label{indm2}
		\end{eqnarray}
		Therefore, combining (\ref{indm1}) and (\ref{indm2}), we get the required bound.
	\end{proof}
	\begin{remark}
		In particular if we put $r=1$ in Corollary \ref{ani1} then it becomes \begin{eqnarray}
			\textbf{ber}^{2}(A)
			&\le& \min\bigg\{ \frac{1}{4}\big\||A^*|^{4}+A^{\dag}A\big\|_{ber}+\frac{1}{2}\textbf{ber}(A^{\dag}A^2A^*),\nonumber\\
			&&\hspace{0.8 cm}\frac{1}{4}\big\||A|^{4}+AA^{\dag}\big\|_{ber}+\frac{1}{2}\textbf{ber}(A^*A^2A^{\dag})\bigg\}.\label{ani011}
		\end{eqnarray}
		In particular if we take $\alpha=1$ in \cite[Th. 2.11]{sen2023bounds} then we get the following bound 
		\begin{eqnarray}
			\textbf{ber}^2(A)\le \frac{1}{4}\textbf{ber}(|A|^4+I)+\frac{1}{2}\textbf{ber}(|A|^2).\label{ani20}
		\end{eqnarray} 
		If we take $ \mathscr{H}=\mathbb{C}^2$ and  $A=\begin{bmatrix}
			\frac{1}{2} & 0\\
			\frac{1}{2} & 0
		\end{bmatrix}$. Then inequality (\ref{ani011}) gives $\textbf{ber}^2(A)\le 0.3125$ whereas inequality (\ref{ani20}) gives $\textbf{ber}^2(A)\le 0.5625$. Therefore inequality  (\ref{ani011}) gives better bound than inequality (\ref{ani20}).
		
	\end{remark}

	Next we obtain the following lemma.

	\begin{lemma}\label{trainvr}
		Let $A_1,A_2,A_3\in  \mathscr{B}( \mathscr{H})$, where $A_1$ and $A_2$ are both positive.  If $\begin{bmatrix}
			A_1 & A^*_3\\
			A_3 & A_2
		\end{bmatrix} \in  \mathscr{B}( \mathscr{H}\oplus  \mathscr{H})$ is positive then for any $r\ge 1$
		\[
		\textbf{ber}^{2r}(A_3)\le \bigg\|\frac{A^{r\gamma}_1}{\gamma}+\frac{A^{r\delta}_2}{\delta}\bigg\|_{ber},
		\] where $\gamma,\delta>1$ and $\frac{1}{\gamma}+\frac{1}{\delta}=1$.
	\end{lemma}
	\begin{proof}
		Let $\hat{k}_x$ be a normalized reproducing kernel of $ \mathscr{H}$. From Lemma \ref{mm}  we get
		\begin{eqnarray*}
			|\langle A_3\hat{k}_x,\hat{k}_x \rangle|^2&\le& \langle A_1\hat{k}_x,\hat{k}_x \rangle \langle A_2\hat{k}_x,\hat{k}_x\rangle\\
			&\le& \bigg(\frac{1}{\gamma}\langle A_1 \hat{k}_x,\hat{k}_x \rangle^{r\gamma}+\frac{1}{\delta} \langle A_2\hat{k}_x,\hat{k}_x\rangle^{r\delta}\bigg)^{\frac{1}{r}}\,\,\,\Big(\mbox{ by  Lemma \ref{cat} }\Big)\\
			&\le& \bigg(\frac{1}{\gamma}\langle A^{r\gamma}_1 \hat{k}_x,\hat{k}_x \rangle+\frac{1}{\delta} \langle A^{r\delta}_2 \hat{k}_x,\hat{k}_x\rangle\bigg)^{\frac{1}{r}}\,\,\,\,\,\,\,\Big(\mbox{ by  Lemma \ref{dog} }\Big)\\
			&=& \bigg(\langle (\frac{A^{r\gamma}_1}{\gamma}+\frac{A^{r\delta}_2}{\delta})\hat{k}_x,\hat{k}_x\rangle\bigg)^{\frac{1}{r}}\\
			&\le &\bigg\|\frac{A^{r\gamma}_1}{\gamma}+\frac{A^{r\delta}_2}{\delta}\bigg\|^{\frac{1}{r}}_{ber}.
		\end{eqnarray*}
		Therefore, taking supremum over all $x\in X$, we get the required inequality.
	\end{proof}
	
	Our next result reads as follows.
	\begin{theorem}
		Let $A_1,A_2\in  \mathscr{C} \mathscr{R}( \mathscr{H})$ and $\lambda\in [0,1]$. Then  for any $r\ge 1$
		\[\textbf{ber}^{2r}(A_1+A_2)\le \bigg
		\|\frac{\big(  \big(A_1A^*_1\big)^{\lambda}+ \big(A^*_2A_2)^{\lambda} \big)^{r\gamma}}{\gamma}+\frac{\big( \big(A^{\dag}_1A_1)^{(1-\lambda)}+\big(A_2A^{\dag}_2\big)^{(1-\lambda)}\big)^{r\delta}}{\delta}\bigg\|_{ber},\]
		\label{trainv}where $\gamma,\delta>1$ and $\frac{1}{\gamma}+\frac{1}{\delta}=1$.
	\end{theorem}
	\begin{proof}
		Since $\begin{bmatrix}
			\big(A_1A^*_1\big)^{\lambda} & A^*_1\\
			A_1 & \big(A^{\dag}_1A_1)^{(1-\lambda)}
		\end{bmatrix}$ and $\begin{bmatrix}
			\big(A^*_2A_2)^{\lambda} & A^*_2\\
			A_2 & \big(A_2A^{\dag}_2\big)^{(1-\lambda)}
		\end{bmatrix}$ are positive thus, the sum of these two block matrices 
		\[\begin{bmatrix}
			\big(A_1A^*_1\big)^{\lambda}+ \big(A^*_2A_2)^{\lambda}  & (A_1+A_2)^*\\
			A_1+A_2 & \big(A^{\dag}_1A_1)^{(1-\lambda)}+\big(A_2A^{\dag}_2\big)^{(1-\lambda)}
		\end{bmatrix}\ge 0.\]
		Therefore, the desired bounds are followed from Lemma \ref{trainvr}.
	\end{proof}
	
	\begin{remark} Taking $A_2=0$, $A_1=A$ and $\frac{1}{\gamma}=\frac{1}{\delta}=\frac{1}{2}$ in Theorem \ref{trainv} we get
		\begin{equation}
			\textbf{ber}^{2r}(A)\le\frac{1}{2} \big
			\| (AA^*)^{2\lambda r}+(A^{\dag}A)^{2(1-\lambda)r}\big\|_{ber}.\label{pinh1}
		\end{equation} In \cite[Cor. 3]{guesba2024berezin} it is obtained that
		
		\begin{equation}
			\textbf{ber}^r(A)\le \frac{1}{2}\bigg\|(AA^*)^{\lambda r}+(A^*A)^{(1-\lambda)r}\bigg\|_{ber}.\label{pinh2}
		\end{equation}
		
		If we take $ \mathscr{H}=\mathbb{C}^2$ and $A=\begin{bmatrix}
			2 & 2\\
			0 & 0
		\end{bmatrix}$, $\lambda=\frac{1}{2}$ and $r=2$  then inequality (\ref{pinh1}) gives $\textbf{ber}^2(A)\le \sqrt{32.25},$ whereas inequality (\ref{pinh2}) gives $\textbf{ber}^2(A)\le 6$.  Therefore inequality (\ref{pinh1})
		gives better bound than that in (\ref{pinh2}).
		
	\end{remark}
	In the following theorem we derive upper bounds of the Berezin norm for the sum of two operators.  
	
	\begin{theorem}\label{problm}
		Let $A_1,A_2\in  \mathscr{CR}( \mathscr{H})$. Then 
		\begin{eqnarray*}
			(i)~~  \|A_1+A_2\|^{2}_{ber} &\le &\|A^*_1A_1+A^*_2A_2\|_{ber}\|A_1A^{\dag}_1+A_2A^{\dag}_2\|_{ber}\\
			(ii)~~  \|A_1+A_2\|^{2}_{ber} &\le &\|A^*_1A_1+A^{\dag}_2A_2\|_{ber}\|A_2A^*_2+A_1A^{\dag}_1\|_{ber}.
		\end{eqnarray*}
	\end{theorem}
	\begin{proof} 
		Let $\hat{k}_x$ and  $\hat{k}_y$ be the normalized reproducing kernels of $ \mathscr{H}$.

		$\begin{bmatrix}
			A^*_1A_1+A^*_2A_2 & (A_1+A_2)^*\\
			A_1+A_2 & A_1A^{\dag}_1+A_2A^{\dag}_2
		\end{bmatrix}$ and $\begin{bmatrix}
			A^*_1A_1+A^{\dag}_2A_2 & (A_1+A_2)^*\\
			A_1+A_2 & A_2A^*_2+A_1A^{\dag}_1
		\end{bmatrix}$ are also positive.
		Now, using Lemma \ref{mm} we get
		\begin{eqnarray*}
			|\langle (A_1+A_2)\hat{k}_x,\hat{k}_y\rangle|^2 &\le& \langle  \big(A^*_1A_1+A^*_2A_2\big)\hat{k}_x,\hat{k}_x\rangle \langle \big(A_1A^{\dag}_1+A_2A^{\dag}_2\big)\hat{k}_y,\hat{k}_y\rangle\\
			&\le& \|A^*_1A_1+A^*_2A_2\|_{ber}\|A_1A^{\dag}_1+A_2A^{\dag}_2\|_{ber}
		\end{eqnarray*}
		and 
		\begin{eqnarray*}
			|\langle (A_1+A_2)\hat{k}_x,\hat{k}_y\rangle|^2 &\le& \langle  \big(A^*_1A_1+A^{\dag}_2A_2\big)\hat{k}_x,\hat{k}_x\rangle \langle \big(A_2A^*_2+A_1A^{\dag}_1\big)\hat{k}_y,\hat{k}_y\rangle\\
			&\le& \|A^*_1A_1+A^{\dag}_2A_2\|_{ber}\|A_2A^*_2+A_1A^{\dag}_1\|_{ber}
		\end{eqnarray*}
		Therefore, we get the required inequalities by taking supremum over all $x,y\in X$. \end{proof}
	
	\begin{remark}
		
		(i)~~ In particular, if we take $A_1=A$, $A_2=A^{\dag}$ and  $A_1=A$, $A_2=\big(A^{\dag}\big)^*$ in the first and second  bound  of Theorem \ref{problm}, respectively, then we get the existing results
		\begin{eqnarray*}
			\|A+A^{\dag}\|^2_{ber}&\le &\|AA^{\dag}+A^{\dag}A\|_{ber}\||A|^2+|A^{\dag}|^2\|_{ber}\\
			&=&\textbf{ber}\big(AA^{\dag}+A^{\dag}A\big)\||A|^2+|A^{\dag}|^2\|_{ber}
		\end{eqnarray*}
		and 
		\begin{eqnarray*}
			\|A+\big(A^{\dag}\big)^*\|^2_{ber}&\le &\||A|^2+A^{\dag}A\|_{ber}\||A^{\dag}|^2+AA^{\dag}\|_{ber}\\    &=&\textbf{ber}\big(|A|^2+A^{\dag}A\big)\textbf{ber}\big(|A^{\dag}|^2+AA^{\dag}\big),
		\end{eqnarray*}
		which are proved in \cite[Th. 2.7 (i),(ii)]{guesba2025further}.\\
		(ii)   Considering $A$ to be invertible,  $A_1=A$, $A_2=A^{-1}$  and $A_1=A$, $A_2=(A^{-1})^*$ in the first and second bound of Theorem \ref{problm}, respectively, then we get the existing inequality given in  \cite[ Cor. 2.8]{guesba2025further}, namely,
		
		$$\|A+A^{-1}\|^2_{ber} \le 2\||A|^2+|A^{-1}|^2\|_{ber}$$
		\mbox{and}
		$$ \|A+\big(A^{-1}\big)^*\|^2_{ber}\le \textbf{ber}\big(|A|^2+I\big)\textbf{ber}\big(|A^{-1}|^2+I\big).$$
	\end{remark}

	In the following Theorems, we develop upper bounds of the Berezin number for product operators.

	\begin{theorem}
		Let $A_1,A_2\in  \mathscr{CR}( \mathscr{H})$. Then for any $r\ge 1$\label{som}
		\[\textbf{ber}^r(A^*_1A_2)\le \frac{1}{2}\min\bigg\{\||A_2|^{2r}+(A^*_1A_2A^{\dag}_2A_1)^r\|_{ber},\||A_1|^{2r}+(A^*_2A_1A^{\dag}_1A_2)^r\|_{ber}\bigg\}.\]
		
	\end{theorem}
	\begin{proof}
		Suppose that $\hat{k}_x$ be a normalized reproducing kernel of $ \mathscr{H}$. Then we get
		\begin{eqnarray*}
			|\langle A^*_1A_2\hat{k}_x,\hat{k}_x\rangle|
			&=&   |\langle A_2\hat{k}_x,A_1\hat{k}_x\rangle|\\
			&\le& \langle |A_2|^2 \hat{k}_x,\hat{k}_x\rangle^{\frac{1}{2}}\langle A_2A^{\dag}_2A_1\hat{k}_x, A_1\hat{k}_x\rangle^{\frac{1}{2}}\,\,\,\,\,\,\,\,\,\,\,\,\,\,\,\,\,\,\,\,\,\,\,\,\,\,\,\Big(\mbox{by Lemma \ref{bandor}}\Big)\\
			&=& \langle |A_2|^2 \hat{k}_x,\hat{k}_x\rangle^{\frac{1}{2}}\langle A^*_1A_2A^{\dag}_2A_1\hat{k}_x,  \hat{k}_x\rangle^{\frac{1}{2}}\\
			&\le& \frac{\langle |A_2|^2 \hat{k}_x,\hat{k}_x\rangle}{2}+\frac{\langle A^*_1A_2A^{\dag}_2A_1\hat{k}_x,  \hat{k}_x\rangle}{2}\\
			&\le& \bigg(\frac{\langle |A_2|^2 \hat{k}_x,\hat{k}_x\rangle^r}{2}+\frac{\langle A^*_1A_2A^{\dag}_2A_1\hat{k}_x,  \hat{k}_x\rangle^r}{2}\bigg)^{\frac{1}{r}}\,\,\,\,\Big(\mbox{by Lemma \ref{cat}}\Big)\\
			&\le& \bigg(\frac{\langle |A_2|^{2r} \hat{k}_x,\hat{k}_x\rangle}{2}+\frac{\langle (A^*_1A_2A^{\dag}_2A_1)^r\hat{k}_x,  \hat{k}_x\rangle}{2}\bigg)^{\frac{1}{r}}\Big(\mbox{by Lemma \ref{dog}}\Big)\\
			&=& \bigg(\frac{1}{2}\big\langle \big(|A_2|^{2r} +(A^*_1A_2A^{\dag}_2A_1)^r\big)\hat{k}_x,\hat{k}_x\big\rangle\bigg)^{\frac{1}{r}}\\
			&\le & \bigg(\frac{1}{2}\||A_2|^{2r}+(A^*_1A_2A^{\dag}_2A_1)^r\|_{ber}\bigg)^{\frac{1}{r}}.
		\end{eqnarray*}
		Taking supremum over all $x\in X$, we get  
		\begin{eqnarray}\textbf{ber}^r(A^*_1A_2)\le \frac{1}{2}\||A_2|^{2r}+(A^*_1A_2A^{\dag}_2A_1)^r\|_{ber}.\label{hari3}
		\end{eqnarray}
		Again, similarly interchanging $A_1$ and $A_2$ we get
		\begin{eqnarray} \textbf{ber}^r(A^*_2A_1)=\textbf{ber}^r(A^*_1A_2)\le \||A_1|^{2r}+(A^*_2A_1A^{\dag}_1A_2)^r\|_{ber}. \label{hari4}\end{eqnarray}
		Combining (\ref{hari3}) and (\ref{hari4}) we obtain the desired upper bound.
	\end{proof}
	\begin{remark} In  particular if we consider $A_1=I$, $A_2=A$ and $r=1$ then from Theorem \ref{som} we get \[
		\textbf{ber}(A)\le \frac{1}{2}\||A|^2+AA^{\dag}\|_{ber}=\frac{1}{2}\textbf{ber}\big(|A|^2+AA^{\dag}\big). 
		\]
		Therefore, it is clear that Theorem \ref{som} generalizes the result proved in \cite[Th. 2.2]{guesba2025further}.
	\end{remark}

	Our next result yields another inequality of the Berezin norm.

	\begin{theorem} Let $A_1,B_1,A_2,B_2\in \mathscr{B(H)}$ and $M,N\in  \mathscr{CR(H)}$. Then for $r,s\ge 1$\label{ram}
		\begin{eqnarray*}
			\|A_1^*MB_1+A^*_2NB_2\|^{2}_{ber}&\le& \frac{1}{2^{r+s-2}}\|(B^*_1|M|^2B_1)^r+(B^*_2|N|^2B_2)^r\|^{\frac{1}{r}}_{ber}\\
			&&\| (A^*_1MM^{\dag} A_1)^s+(A^*_2NN^{\dag}A_2)^s\|^{\frac{1}{s}}_{ber}.    
		\end{eqnarray*}
	\end{theorem}
	\begin{proof} Let $\hat{k}_x$ and  $\hat{k}_y$ be the normalized reproducing kernels of $ \mathscr{H}$.
		\begin{eqnarray*}
			&&|\langle (A^*_1MB_1+A^*_2NB_2)\hat{k}_x,\hat{k}_y \rangle|^2\\
			&\le&\big(|\langle A^*_1MB_1\hat{k}_x,\hat{k}_y \rangle|+|\langle A^*_2NB_2 \hat{k}_x,\hat{k}_y\rangle|\big)^2\\
			&=&\big(|\langle MB_1\hat{k}_x,A_1\hat{k}_y \rangle|+|\langle NB_2\hat{k}_x,A_2\hat{k}_y\rangle|\big)^2\\
			&\le&\bigg(\langle |M|^2B_1\hat{k}_x,B_1\hat{k}_x\rangle^{\frac{1}{2}} \langle MM^{\dag} A_1\hat{k}_y       ,A_1\hat{k}_y \rangle^{\frac{1}{2}}\\
			&&\,\,\,\,\,\,\,\,+\langle |N|^2B_2 \hat{k}_x,B_2\hat{k}_x\rangle^{\frac{1}{2}}\langle NN^{\dag}A_2\hat{k}_y,
			A_2\hat{k}_y\rangle^{\frac{1}{2}}\bigg)^2\,\,\,\Big(\mbox{ by Lemma \ref{bandor}}\Big)\\
			&\le& \bigg( \langle |M|^2B_1\hat{k}_x,B_1\hat{k}_x\rangle+ \langle |N|^2B_2 \hat{k}_x,B_2 \hat{k}_x\rangle\bigg)\\
			&&\,\,\,\,\,\,\,\,\bigg(\langle MM^{\dag} A_1\hat{k}_y      ,A_1\hat{k}_y \rangle+\langle NN^{\dag}A_2\hat{k}_y,
			A_2\hat{k}_y\rangle\bigg)\,\,\,\,\,\,\,\,\,\, \Big(\mbox{by (\ref{sree})}\Big)\\
			&=&4 \bigg( \frac{\langle |M|^2B_1\hat{k}_x,B_1\hat{k}_x\rangle}{2}+ \frac{\langle |N|^2B_2 \hat{k}_x,B_2 \hat{k}_x\rangle}{2}\bigg)\\
			&&\,\,\,\,\,\,\,\,\bigg(\frac{\langle MM^{\dag} A_1\hat{k}_y       ,A_1\hat{k}_y \rangle}{2}+\frac{\langle NN^{\dag}A_2\hat{k}_y,
				A_2\hat{k}_y\rangle}{2}\bigg)\\
			&\le&4 \bigg( \frac{\langle B^*_1|M|^2B_1\hat{k}_x,\hat{k}_x\rangle^r}{2}+ \frac{\langle B^*_2|N|^2B_2 \hat{k}_x, \hat{k}_x\rangle^r}{2}\bigg)^{\frac{1}{r}}\\&&\bigg(\frac{\langle A^*_1MM^{\dag} A_1\hat{k}_y      ,\hat{k}_y \rangle^s}{2}+\frac{\langle A^*_2NN^{\dag}A_2\hat{k}_y,
				\hat{k}_y\rangle^s}{2}\bigg)^{\frac{1}{s}}\,\,\,\Big(\mbox{ by Lemma \ref{cat}}\Big)\\
			&\le& 4\bigg( \langle \big(\frac{(B^*_1|M|^2B_1)^r+(B^*_2|N|^2B_2)^r}{2}\big)\hat{k}_x,\hat{k}_x\rangle\bigg)^{\frac{1}{r}}\\
			&& \,\,\,\,\,\,\,\,\bigg(\big\langle \big(\frac{(A^*_1MM^{\dag} A_1)^s+(A^*_2NN^{\dag}A_2)^s}{2}\big)\hat{k}_y     ,\hat{k}_y\big \rangle\bigg)^{\frac{1}{s}}\,\,\,\Big(\mbox{ by Lemma \ref{dog}}\Big)\\
			&\le& \frac{1}{2^{r+s-2}}\|(B^*_1|M|^2B_1)^r+(B^*_2|N|^2B_2)^r\|^{\frac{1}{r}}_{ber}\| (A^*_1MM^{\dag} A_1)^s+(A^*_2NN^{\dag}A_2)^s\|^{\frac{1}{s}}_{ber}.
		\end{eqnarray*}
		Therefore, we get the desired inequalities taking the supremum over all $x,y\in X$.
	\end{proof}
	The following corollary is derived by taking $A_1=B_1=A_2=B_2=I$ in Theorem \ref{ram}.
	\begin{cor}  Let $M,N\in  \mathscr{CR(H)}$. Then for $r,s\ge 1$\label{ram}
		\begin{eqnarray*}
			\|M+N\|^{2}_{ber}&\le& \frac{1}{2^{r+s-2}}\||M|^{2r}+|N|^{2r}\|^{\frac{1}{r}}_{ber}\| (MM^{\dag} )^s+(NN^{\dag})^s\|^{\frac{1}{s}}_{ber}.         
	\end{eqnarray*}\end{cor}
	
	\begin{remark}
		In particular, if we take $M=N=I$ in  Theorem \ref{ram} then it becomes
		\begin{eqnarray}
			\|A^*_1B_1+A^*_2B_2\|^{2}_{ber}\le \frac{1}{2^{r+s-2}}\||B_1|^{2r}+|B_2|^{2r}\|^{\frac{1}{r}}_{ber}\| | A_1|^{2s}+|A_2|^{2s}\|^{\frac{1}{s}}_{ber}. \label{ram33}  \end{eqnarray}
		In \cite[Th. 3.5]{garayev2019vcebyvsev} it was proved that if $A_1,B_1,A_2, B_2\in  \mathscr{B(H)}$ then
		\begin{eqnarray*}
			\textbf{ber}^2\big(A^*_1B_1+A^*_2B_2\big)&\le&\frac{1}{2^{r+s-2}}\textbf{ber}^{\frac{1}{r}}\big(|B_1|^{2r}+|B_2|^{2r}\big)\textbf{ber}^{\frac{1}{s}} \big(|A_1|^{2s}+|A_2|^{2s}\big), \end{eqnarray*} which follows form inequality (\ref{ram33}).
	\end{remark}

	\begin{theorem}
		Let $A_i, B_i\in \mathscr{B(H)}$, $i=1,\cdots,n$ and $M_i\in \mathscr{C} \mathscr{B(H)}$. Then for any $r\ge 1$ \label{ram44}
		\[\textbf{ber}^r\bigg(\sum_{i=1}^n A^*_iM_iB_i\bigg)\le\frac{n^{r-1}}{2^r}\bigg\| \sum_{i=1}^n\big(B^*_i|M_i|^{2}B_i+ A^*_iM_iM^{\dag}_i A\big)^r\bigg\|_{ber}.\]
	\end{theorem}
	\begin{proof} Let $\hat{k}_x$ be a normalized reproducing kernel of $ \mathscr{H}$. Then
		\begin{eqnarray*}
			&&|\langle (\sum_{i=1}^n (A^*_iM_iB_i)\hat{k}_x,\hat{k}_x\rangle|^r\\
			&= &|\sum_{i=1}^n \langle A^*_iM_iB_i\hat{k}_x,\hat{k}_x\rangle|^r\\
			&\le &\bigg(\sum_{i=1}^n| \langle A^*_iM_iB_i\hat{k}_x,\hat{k}_x\rangle|\bigg)^r\\
			&\le &n^{r-1}\bigg(\sum_{i=1}^n| \langle A^*_iM_iB_i\hat{k}_x,\hat{k}_x\rangle|^r\bigg)\,\,\,\Big(\mbox{ by Lemma \ref{mit}}\Big)\\
			&=&n^{r-1}\bigg(\sum_{i=1}^n| \langle M_iB_i\hat{k}_x,A_i\hat{k}_x\rangle|^r\bigg)\\
			&\le&n^{r-1}\bigg(\sum_{i=1}^n \big(\langle M^*_iM_iB_i\hat{k}_x,B_i\hat{k}_x\rangle^{\frac{1}{2}} \langle M_iM^{\dag}_i A_i\hat{k}_x,A_i\hat{k}_x\rangle^{\frac{1}{2}}\big)^r\bigg)\\
			&\le&\frac{n^{r-1}}{2^r}\bigg(\sum_{i=1}^n \big(\langle (B^*_i|M_i|^{2}B_i)\hat{k}_x,\hat{k}_x\rangle+ \langle (A^*_iM_iM^{\dag}_i A)\hat{k}_x,\hat{k}_x\rangle\big)^r\bigg) \\
			&=&\frac{n^{r-1}}{2^r}\big \langle \sum_{i=1}^n(B^*_i|M_i|^{2}B_i+A^*_iM_iM^{\dag}_i A)^r\hat{k}_x,\hat{k}_x\big\rangle\,\,\,\,\,\,\Big(\mbox{ by Lemma \ref{dog}}\Big)\\
			&\le&\frac{n^{r-1}}{2^r}\bigg\| \sum_{i=1}^n\big(B^*_i|M_i|^{2}B_i+ A^*_iM_iM^{\dag}_i A\big)^r\bigg\|_{ber}.\\
		\end{eqnarray*}
		Hence, taking supremum over all $x\in X$, we get the desired upper bound.
	\end{proof}

	\begin{remark} 
		(i)~~In particular, if we consider $M=I$ and $A^*=A$ in Theorem \ref{ram44}, then the inequality becomes 
		\[\textbf{ber}^r\bigg(\sum_{i=1}^n A_iB_i\bigg)\le \frac{n^{r-1}}{2^r}\|\sum_{i=1}^{n}\big(B_i^*B_i+A_iA_i^*\big)^r\|_{ber},\]
		which was obtained in \cite[Cor. 2.28]{bhunia2024berezin}.\\
		(ii)~~ In particular, if we consider $B=I$, $A=I$, $r=1$  and $n=1$ in Theorem \ref{ram44}, then the inequality becomes 
		\[ \textbf{ber}(M)\le \frac{1}{2}\||M|^2+MM^{\dag}\|_{ber} =\frac{1}{2}\textbf{ber}(|M|^2+MM^{\dag})\]
		which was obtained in \cite[Th. 2.2]{guesba2025further}.
	\end{remark}

	\textit{Acknowledgements.} Mr. Saikat Mahapatra would like to thank the UGC, Govt. of India for the financial support (NTA Ref. No. 211610170555) in the form of a Senior Research Fellowship under the mentorship of Dr. Riddhick Birbonshi. Mr Anirban Sen would like to thank the CSIR, Govt. of India, for the financial support in the form of a Senior Research Fellowship under the mentorship of Prof. Kallol Paul.

	\bibliographystyle{amsplain}

\begin{thebibliography}{99}
		
		
		
		\bibitem{BY_JIA_2020} M. Bakherad and U. Yamanc\i, New estimations for the Berezin number inequality, J. Inequal. Appl., (2020), Paper No. 40, 9 pp.
		
		\bibitem{AB_02} A. Ben-Israel, The Moore of the Moore-Penrose inverse, Electron. J. Linear Algebra., 9 (2002), 150--157.
		
		\bibitem{BER1} F.A. Berezin,  Quantization, Math. USSR-Izv., 8 (1974), 1109--1163.
		
		
		\bibitem{BER2} F.A. Berezin, Covariant and contravariant symbols for operators, Math. USSR-Izv., 6 (1972), 1117--1151.
		
		\bibitem{bhunia2024berezin} P. Bhunia, A. Sen, S. Barik, and K. Paul, Berezin number and Berezin norm inequalities for operator matrices. Linear and Multilinear Algebra, 72(16), (2024), 2749--2768. 
		
		\bibitem{NFAO} P. Bhunia, M. Gurdal, K. Paul, A. Sen and R. Tapdigoglu, On a new norm on the space of reproducing kernel Hilbert space operators and Berezin radius inequalities, Numer. Funct. Anal. Optim., 44 (2023), no. 9, 970--986.
		
		\bibitem{bhunia2023development} P. Bhunia, A. Sen and K. Paul, Development of the Berezin number inequalities, Acta Math. Sin. (Engl. Ser.), 39 (2023), no. 7, 1219--1228.
		
		\bibitem{BPS_CAOT_23} P. Bhunia, K. Paul and A. Sen, Inequalities involving Berezin norm and
		Berezin number, Complex Anal. Oper. Theory, 17 (2023), no. 1, Paper No. 7, 15 pp.
		
		\bibitem{buzano1971generalizzazione}  M.L. Buzano, Generalizzatione della disuguaglianza di Cauchy-Schwarz, Rend. Semin. Mat. Univ. Politech. Torino, 31(1971/73) (1974), 405--409.
		
		
		
		
		
		
		\bibitem{GA_CAOT_2021} M.T. Garayev and M.W. Alomari, Inequalities for the Berezin number of operators and related questions, Complex Anal. Oper. Theory, 15 (2021), no. 2, Paper no. 30, 30 pp.
		
		\bibitem{garayev2019vcebyvsev} M.T. Garayev and U. Yamanc\i, \v Ceby\v sev’s type inequalities and power inequalities for the Berezin number of operators, Filomat, 33 (2019), no. 8, 2307--2316.
		
		\bibitem{guesba2025further} M. Guesba, S. Barik and K. Paul, Further berezin number
		and berezin norm inequalities for sums and products of operators, Complex Anal. and Oper. Theory, 19, 32 (2025), 21pp.
		
		\bibitem{guesba2024berezin} M. Guesba, P. Bhunia, and K. Paul, Berezin number inequalities via positivity of $2 \times 2$ block matrices, Oper. Matrices, 18(1) (2024), 83--95.
		
		\bibitem{HLB_LAMA_2020} M. Hajmohamadi, R. Lashkaripour and M. Bakherad, Improvements of Berezin number inequalities, Linear Multilinear Algebra, 68 (2020), no. 6, 1218--1229.
		
		\bibitem {halmos2012hilbert} P.R. Halmos, A Hilbert space problems book, Springer Verlag, New York, (1982).
		
		\bibitem{hardy1952inequalities} G.H. Hardy, J.E. Littlewood and G. Pólya, Inequalities, 2nd ed., Cambridge Univ. Press, Cambridge, (1988).
		
		\bibitem{huban2022some} M.B. Huban, H. Ba\c{s}aran and M. G\"{u}rdal, Some New Inequalities via Berezin Numbers, Turk. J. Math. Comput. Sci., 14(1), (2022), 129--137.
		
		
		
		
		\bibitem{K_CAOT_2013} M.T. Karaev, Reproducing kernels and Berezin symbols techniques in various questions of operator theory, Complex Anal. Oper. Theory, 7(4) (2013), 983--1018.
		
		\bibitem{KI_MFAT_2013} M.T. Karaev and N. Sh. Iskenderov, Berezin number of operators and related questions, Methods Funct. Anal. Topology, 19 (2013), no. 1, 68--72.
		
		
		
		\bibitem{KAR_JFA_2006} M.T. Karaev, Berezin symbol and invertibility of operators on the functional Hilbert spaces, J. Funct. Anal., 238 (2006), 181--192.
		
		\bibitem{kittaneh1988notes} F. Kittaneh, Notes on some inequalities for Hilbert space operators, Publ. Res. Inst. Math. Sci., 24 (1988), no. 2, 283--293.
		
		\bibitem{MMM_JMAA_23}	S. Majee, A. Maji and A. Manna, Numerical radius and Berezin number inequality,	J. Math. Anal. Appl., 517 (2023), no. 1, Paper No. 126566, 20 pp.
		
			\bibitem{RR_book_2007} A.R. Martínez-Avendaño and P. Rosenthal, An introduction to operators on the Hardy-Hilbert space, Grad. Texts in Math., 237, Springer, New York, (2007).
		
		
		
		\bibitem{PR_Book_2016} V.I. Paulsen and M. Raghupati, An introduction to the theory of reproducing kernel Hilbert spaces, Cambridge Univ. Press, (2016).
		
		\bibitem{furuta2005mond} J. Pe{\v{c}}ari{\'c}, Josip, T. Furuta, J. Mi{\'c}i{\'c} and Y. Seo, Mond-Pe{\v{c}}ari{\'c} method in operator inequalities method in operator inequalities.
		Inequalities for bounded selfadjoint operators on a Hilbert space
		Monogr. Inequal., 1 Element, Zagreb, (2005).
		
		\bibitem{Pen_55} R. Penrose, A generalized inverse for matrices, Math. Proc. Cambridge Philos. Soc., 51 (1955), 406--413.
		
		
		\bibitem{sababheh2024numerical} M. Sababheh, D.S. Djordjevi\'c and H.R. Moradi, Numerical Radius and Norm Bounds via the Moore-Penrose Inverse, Complex Anal. Oper. Theory, 18 (2024), no. 5, Paper No. 117, 11 pp.
		
		
		
		
		
		\bibitem{sen2025berezin} A. Sen and K. Paul, Berezin number and numerical radius inequalities, Vietnam J. Math., 53(2), (2025), 277--289.
		
		\bibitem{sen2023bounds} A. Sen, P. Bhunia and K. Paul, Bounds for the Berezin number of reproducing kernel Hilbert space operators, Filomat, 37 (2023), no. 6, 1741--1749.
		
		\bibitem{RMS} A. Sen, P. Bhunia and K. Paul, Berezin number inequalities of operators on reproducing kernel Hilbert spaces, Rocky Mountain J. Math., 52 (2022), no. 3, 1039--1046.
		
		
		
		
		\bibitem {vasic1971some} M.P. Vasi\'c and D.J. Ke\^cki\'c, Some inequalities for complex
		numbers, Math. Balkanica, 1 (1971), 282--286.
		
		\bibitem {YI_LAMA_2022}  U. Yamancı and İ.M. Karlı, Further refinements of the Berezin number inequalities on operators, Linear Multilinear Algebra, 70 (2022), no. 20, 5237--5246.
		
		\bibitem {YG_NYJM_2017} U. Yamanc\i\ and M. G\"{u}rdal, On numerical radius and
		Berezin number inequalities for reproducing kernel Hilbert space, New York J.
		Math., 23 (2017), 1531--1537.
		
		\bibitem{ZHU_book_2007} K. Zhu, Operator Theory in Function Spaces, 2nd ed., Mathematical Surveys and Monographs, vol. 138, Amer. Math. Soc., Providence, RI, (2007).
		
		
		
		
		
		
		
		
		
		
		
		
		
		
		
		
		
		
		
		
		
		
		
		
		
		
		
		
		
		
		
	\end{thebibliography}

\end{document}